\documentclass{article}
\newcommand\mut[1]{\ignorespaces}

\usepackage{amssymb}
\usepackage{amsmath}
\usepackage{amsthm}
\newtheorem{theorem}{Theorem}[section]
\newtheorem{lemma}[theorem]{Lemma}
\newtheorem{proposition}[theorem]{Proposition}
\newtheorem{remark}[theorem]{Remark}

\usepackage[usenames]{color}

\title{On the Geil-Matsumoto Bound and the Length of AG Codes} 
\author{Maria Bras-Amor\'os, Albert Vico-Oton}
\date{Designs, Codes and Cryptography, Springer, \\vol. 70, n. 1-2, pp. 117-125, January 2014.}

\begin{document}
\maketitle

\begin{abstract}
The Geil-Matsumoto bound conditions the number of rational places of a function field in terms of the Weierstrass semigroup of any of the places. Lewittes' bound  preceded the Geil-Matsumoto bound and it only considers the smallest generator of the numerical semigroup. It can be derived from the Geil-Matsumoto bound and so it is weaker. However, for general semigroups the Geil-Matsumoto bound does not have a closed formula and it may be hard to compute, while Lewittes' bound is very simple. We give a closed formula for the Geil-Matsumoto bound for the case when the Weierstrass semigroup has two generators. We first find a solution to the membership problem for semigroups generated by two integers and then apply it to find the above formula. We also study the semigroups for which Lewittes's bound and the Geil-Matsumoto bound coincide. We finally investigate on some simplifications for the computation of the Geil-Matsumoto bound.
\end{abstract}

\noindent
{\bf Key words:}
Algebraic function field, Weierstrass semigroup, Geil-Matsumoto bound, gonality bound, Lewittes' bound.

\section{Introduction}
\mut{
Given $n$ pairwise distinct elements $\alpha_1,\dots,\alpha_n$ of a finite field ${\mathbb F}_q$,
the Reed-Solomon code $RS_{\alpha_1,\dots,\alpha_n}(k)$ is defined by
$\{(f(\alpha_1),\dots,f(\alpha_n)): f\in {\mathbb F}_q[x], \deg(f)<k\}.$
Then, the length of $RS_{\alpha_1,\dots,\alpha_n}(k)$ is $n$ and it is bounded by the field size $q$.

A generalization of this is that of Reed-Muller codes. That is,
given $n$ pairwise distinct affine points $P_1,\dots,P_n$ of the affine space ${\mathbb F}_q^m$,
the Reed-Muller code $RM_{P_1,\dots,P_n}(k)$ is defined by
$\{(f(P_1),\dots,f(P_n)): f\in {\mathbb F}_q[x_1,\dots,x_m], \deg(f)<k\}.$
Then, the length of $RM_{P_1,\dots,P_n}(k)$ is $n$ and it is bounded by the size of ${\mathbb F}_q^m$, that is, $q^m$.

Algebraic geometry codes generalize this giving codes attaining very important asymptotic bounds.}

Given $n$ pairwise distinct places $P_1,\dots,P_n$ of degree one of 
an algebraic function field $F/{\mathbb F}_q$,
and a divisor $G$ with support disjoint from $\{P_1,\dots,P_n\}$,
the geometric Goppa code $C_{P_1,\dots,P_n}(G)$ is defined by
$\{(f(P_1),\dots,f(P_n)): f\in L(G)\}.$
See \cite{6:Stichtenoth} for a general reference.
Then, the length of $C_{P_1,\dots,P_n}(G)$ is $n$ and it is bounded by the 
number of places of degree one of $F/{\mathbb F}_q$.
Thus, an important problem of algebraic coding theory is bounding the number of places of degree one of function fields.

The Hasse-Weil bound for the number of places of degree one of  a function field as well as Serre's improvement use only the genus of the function field and the field size.
Geil and Matsumoto give in \cite{6:GeilMatsumoto} a bound 
in terms of the Weierstrass semigroup 
of a rational place (i.e. the set of pole orders of rational functions having only poles in that place). It is a neat formula although it is not closed and it may be computationally hard to calculate. 
Lewittes' bound \cite{6:Lewittes}, also called the gonality bound, preceded the Geil-Matsumoto bound and it
only considers the smallest generator of the numerical semigroup. It can
be derived from the Geil-Matsumoto bound and so it is weaker. 
The advantage of Lewittes' bound with respect to the Geil-Matsumoto bound is that Lewittes' bound is very simple to compute. 

Important curves such as hyperelliptic curves, Hermitian curves or Geil's norm-trace curves \cite{6:Geil:normtrace} have Weierstrass semigroups generated by two integers.
Also,  
for any numerical semigroup $\Lambda$ generated by two coprime integers, one can get the equation of a curve having a place whose Weierstrass semigroup is $\Lambda$ \cite{6:HoLiPe}.

In Section~\ref{6:sec:membership}, we give some notions on numerical semigroups
and solve the membership problem for numerical semigroups generated by two coprime integers.  Then in Section~\ref{6:sec:gm2} we use the result in Section~\ref{6:sec:membership} to deduce a closed formula for the Geil-Matsumoto bound when the Weierstrass semigroup is generated by two integers.
In Section~\ref{6:sec:L-GM} we return to semigroups generated by any number of integers and study
in which cases Lewittes' bound and the Geil-Matsumoto
bound coincide. In Section~\ref{6:sec:simplify} 
we give a result that may simplify the computation
of the Geil-Matsumoto bound.

\section{Membership for semigroups with two generators}
\label{6:sec:membership}
Let ${\mathbb N}_0$ be the set of non-negative integers.
A {\it numerical semigroup} is a subset of ${\mathbb N}_0$ containing $0$, closed under addition and with finite complement in ${\mathbb N}_0$. A general reference for numerical semigroups is \cite{6:RosalesGarciaSanchez}.
For a numerical semigroup $\Lambda$
define
the {\it genus} of $\Lambda$ as the number
$g=\#({\mathbb N}_0\setminus\Lambda)$.
The elements in $\Lambda$ are called the {\it non-gaps} of $\Lambda$
while the elements in
${\mathbb N}_0\setminus\Lambda$ are called the {\it gaps} of $\Lambda$.

The {\it generators} of a numerical semigroup are those non-gaps which can not
be obtained as a sum of two smaller non-gaps.
If $a_1,\dots,a_l$ 
are the generators of a semigroup $\Lambda$ then 
$\Lambda=\{n_1a_1+\dots+n_la_l:n_1,\dots,n_l\in{\mathbb N}_0\}$ and so 
$a_1,\dots,a_l$ are necessarily coprime.
If $a_1,\dots,a_l$ are coprime, we call $\{n_1a_1+\dots+n_la_l:n_1,\dots,n_l\in{\mathbb N}_0\}$
the {\it semigroup generated} by $a_1,\dots,a_l$ and
denote it by $\langle a_1,\dots,a_l\rangle$.

Among numerical semigroups, those generated by two integers, that is, numerical semigroups of the form $\{ma+nb:m,n\in{\mathbb N}_0\}$ for some coprime integers $a,b$, have a particular interest. Important curves such as hyperelliptic curves, Hermitian curves or Geil's norm-trace curves \cite{6:Geil:normtrace} have Weierstrass semigroups generated by two integers.
Properties of semigroups generated by two coprime integers can be found in \cite{6:KiPe}. 
For instance, the semigroup generated by $a$ and $b$ has genus $\frac{(a-1)(b-1)}{2}$, and any element $i\in\Lambda$ can be uniquely written as $i=ma+nb$ with $m,n$ integers such that $0\leqslant n \leqslant a-1$.
From the results in \cite[Section 3.2]{6:HoLiPe} one can get, for any numerical semigroup $\Lambda$ generated by two coprime integers the equation of a curve having a point whose Weierstrass semigroup is $\Lambda$.

For a numerical semigroup, the membership problem is that of determining, for any integer $i$ whether it belongs or not to the numerical semigroup. 
In the next lemma we first state a result already proved in \cite{6:KiPe} and then we give a solution to the membership problem for semigroups generated by two coprime integers. By $x \mod a$ with $x$, $a$ integers we mean the smallest positive integer congruent with $x$ modulo $a$.   

\begin{lemma}
\label{6:lemma:membership}
Suppose $\Lambda$ is generated by $a,b$ with $a<b$.
Let $c$ be the inverse of $b$ modulo $a$.
\begin{enumerate}
\item Any $i\in\Lambda$ can be uniquely written as 
$i=ma+nb$ for some $m,n\geqslant 0$ with $n\leqslant a-1$.
\item $i\in\Lambda$ if and only if $b(i\cdot c\mod a)\leqslant i.$


\end{enumerate}
\end{lemma}

\begin{proof}
\begin{enumerate}
\item
Suppose $i\in\Lambda$. Then 
$i=\tilde ma+\tilde nb$ for some non-negative integers
$\tilde m, \tilde n$.
Let $n=\tilde n \mod a$ and $m=\tilde m+b\lfloor\frac{\tilde n}{a}\rfloor$.
Then 
$i=\tilde ma+\tilde nb=\tilde ma+(a\lfloor\frac{\tilde n}{a}\rfloor+(\tilde n\mod a))b=ma+nb$ with obviously $m,n\geqslant 0$ and $n\leqslant a-1$.

For uniqueness, suppose $i=ma+nb$ for some $m,n\geqslant 0$ and $n\leqslant a-1$, and simultaneously, $i=m'a+n'b$ for some $m',n'\geqslant 0$ and $n'\leqslant a-1$. 
Then $(m-m')a=(n'-n)b$. Since $a$ and $b$ are coprime, $a$ must divide $n-n'$ which can only happen if $n=n'$ and so $m=m'$.
\item
If $i\in\Lambda$ then by the previous statement there exist unique integers $m,n\geqslant 0$ with $n\leqslant a-1$ such that $i=ma+nb$.
In this case, $i\cdot c\mod a = (ma+nb)\cdot c\mod a=n$ and then it is obvious that $b(i\cdot c\mod a)\leqslant i$.

On the other hand, suppose $i\in{\mathbb N}_0$ and define $n=i\cdot c\mod a$.
Then $i-nb$ is a multiple of $a$
since $(i-nb) \mod a =  ((i\mod a) - (nb \mod a))\mod a=0$.
If $nb\leqslant i$ then $i-nb$ is a positive multiple of $a$, say $ma$,
and $i=ma+nb$, so $i\in\Lambda$.

\end{enumerate}
\end{proof}

\begin{remark}
Notice that for the case $b=a+1$
the condition $b(i\cdot c\mod a)\leqslant i$
is equivalent to $(a+1)(i\mod a)\leqslant i$ and to
$a(i\mod a)\leqslant i-(i\mod a)$ and so $i\mod a\leqslant \lfloor\frac{i}{a}\rfloor$.
Therefore, $i\in\langle a, a+1\rangle$ if and only if the remainder of the division of $i$ by $a$ is at most its quotient. This was already proved in \cite{6:GarciaRosales:intervals}.
\end{remark}

\section{The Geil-Matsumoto bound}
\label{6:sec:gm2}

Let $N_q(g)$ be the maximum number of rational places of degree one of a function field over ${\mathbb F}_q$ with genus $g$. 
The Hasse-Weil bound \cite[Theorem V.2.3]{6:Stichtenoth} states $|N_q(g)-(q+1)|\leqslant 2 g \sqrt{q}$. Serre's refinement \cite[Theorem V.3.1]{6:Stichtenoth} 
states
$|N_q(g)-(q+1)|\leqslant g \lfloor 2\sqrt{q}\rfloor$. This means that either 
$N_q(g)\leqslant q+1$ 
or 
\begin{equation}
\label{eq:serre}
N_q(g)\leqslant S_q(g):=q+1+g \lfloor 2\sqrt{q}\rfloor.
\end{equation}
Since $g \lfloor 2\sqrt{q}\rfloor\geqslant 0$ it is enough to state
equation~\ref{eq:serre}.

If we consider the Weierstrass semigroup $\Lambda$ of any such places then we can define $N_q(\Lambda)$ as the maximum number of rational places of degree one of a function field over ${\mathbb F}_q$ such that the Weierstrass semigroup at one of the places is $\Lambda$. 
Lewittes' bound \cite{6:Lewittes} states, if $\lambda_1$ is the first non-zero element in $\Lambda$, 
$$N_q(\Lambda)\leqslant L_q(\Lambda):=q\lambda_1+1$$ and the Geil-Matsumoto bound \cite{6:GeilMatsumoto} is
\begin{equation}\label{6:gm}N_q(\Lambda)\leqslant GM_q(\Lambda):=\#(\Lambda\setminus\cup_{\lambda_i\mbox{ generator of }\Lambda}(q\lambda_i+\Lambda))+1.\end{equation}
In \cite{6:GeilMatsumoto,6:HoLiPe} the next result is proved, from which Lewittes' bound can be deduced from the Geil-Matsumoto bound.

\begin{lemma}
\label{6:l:hlpgm}
$\#(\Lambda\setminus(q\lambda_1+\Lambda))=q\lambda_1$.
\end{lemma}

Here, for a numerical semigroup generated by two coprime integers $a,b$ we describe the Geil-Matsumoto bound in terms of $a, b$ giving a formula which is simpler to compute than (\ref{6:gm}).

\enlargethispage{2cm}
\begin{theorem}
The Geil-Matsumoto bound 
for the semigroup generated by $a$ and $b$ with $a<b$ is
\begin{eqnarray}
\label{6:gm0}
GM_q(\langle a,b\rangle)&=&1+\sum_{n=0}^{a-1}\min\left(q,\left\lceil\frac{q-n}{a}\right\rceil\cdot b\right)
\\
\label{6:gm1}
&=&
\left\{\begin{array}{ll}
1+qa &\mbox{ if }q\leqslant \lfloor\frac{q}{a}\rfloor b\\
1+(q\,\text{\rm mod}\,a)q+(a-(q\,\text{\rm mod}\,a))\lfloor\frac{q}{a}\rfloor b &\mbox{ if } 
\lfloor\frac{q}{a}\rfloor b
<q\leqslant\lceil\frac{q}{a}\rceil b\\
1+ab\lceil\frac{q}{a}\rceil -(a-(q\,\text{\rm mod}\,a))b &\mbox{ if } q>\lceil\frac{q}{a}\rceil b\\
\end{array}
\right.
\end{eqnarray}
\end{theorem}

\begin{proof}
The Geil-Matsumoto bound 
for the semigroup generated by $a$ and $b$ with $a<b$ is
$1+\#\left\{i\in\Lambda:\begin{array}{l}
i-qa\not\in\Lambda\\
i-qb\not\in\Lambda
\end{array}
\right\}.$
By Lemma~\ref{6:lemma:membership} $i\in \Lambda$ if and only if 
$b(i c\mod a)\leqslant i$, where $c$ is the inverse of $b$ modulo $a$.
Now, suppose that $i\in\Lambda$ can be expressed as $i=ma+nb$ 
for some integers 
$m,n\geqslant 0$, $n\leqslant a-1$.
Then 
\begin{eqnarray*}
i-qa\not\in\Lambda &\Longleftrightarrow&
b((i-qa) c\mod a)> i-qa
\\&\Longleftrightarrow&
b((ma+nb-qa) c\mod a)> i-qa
\\&\Longleftrightarrow& 
b(n b c\mod a)> i-qa
\\&\Longleftrightarrow& 
b n > i-qa 
\\&\Longleftrightarrow&
b n > (m-q)a+nb
\\&\Longleftrightarrow&
(m-q)a <0 
\\&\Longleftrightarrow&
m < q
\end{eqnarray*}

\clearpage

\begin{eqnarray*}
i-qb\not\in\Lambda &\Longleftrightarrow&
b((i-qb) c\mod a)> i-qb
\\&\Longleftrightarrow&
b((ma+nb-qb) c\mod a)> i-qb
\\&\Longleftrightarrow& 
b((n-q)b c\mod a)> i-qb
\\&\Longleftrightarrow& 
b((n-q)\mod a)> i-qb 
\\&\Longleftrightarrow&
b((n-q)\mod a)> ma+(n-q)b
\\&\Longleftrightarrow&
b[((n-q)\mod a)-(n-q)]> ma 
\\&\Longleftrightarrow&
b\left(-\left\lfloor\frac{n-q}{a}\right\rfloor a\right)> ma
\\&\Longleftrightarrow&
b\left(\left\lceil-\frac{n-q}{a}\right\rceil \right)> m
\\&\Longleftrightarrow&
b\left\lceil\frac{q-n}{a}\right\rceil> m
\end{eqnarray*}

Consequently, 
the Geil-Matsumoto bound is
\begin{equation*}
1+\sum_{n=0}^{a-1}\min\left(q,\left\lceil\frac{q-n}{a}\right\rceil\cdot b\right)\end{equation*}
Now some technical steps lead to the next formula.
\begin{equation}
\label{6:gm2}
GM_q(\langle a,b\rangle)=\left\{\begin{array}{ll}
1+qa &\mbox{ if }q\leqslant \lceil\frac{q-a+1}{a}\rceil b\\
1+(q\mod a)q+(a-(q\mod a))\lceil\frac{q-a+1}{a}\rceil b &\mbox{ if } 
\lceil\frac{q-a+1}{a}\rceil b
<q\leqslant\lceil\frac{q}{a}\rceil b\\
1+ab\lceil\frac{q}{a}\rceil -(a-(q\mod a))b &\mbox{ if } q>\lceil\frac{q}{a}\rceil b\\
\end{array}\right.
\end{equation}
Since 
$\lceil\frac{q-a+1}{a}\rceil$ 
is the unique integer between 
$\frac{q-a+1}{a}$ and 
$\frac{q}{a}$,
one has 
$\lceil\frac{q-a+1}{a}\rceil=\lfloor\frac{q}{a}\rfloor$, and 
the formula in (\ref{6:gm2}) coincides with that in (\ref{6:gm1}).
\end{proof}

\section{Coincidences of Lewittes's and the Geil-Matsumoto bound}
\label{6:sec:L-GM}

We are interested now in the coincidences of Lewittes's and the Geil-Matsumoto bound.
To get an idea, one can see in Table~\ref{t:LGM} the portion of semigroups for which they coincide for several values of the genus and the field size. 

Beelen and Ruano proved in \cite[Proposition 9]{6:BeelenRuano} that if $q\in\Lambda$ then the bounds coincide. 
For the case of two generators, from equation (\ref{6:gm0}) we deduce that 
$GM_q(\langle a,b\rangle)=L_q(\langle a,b\rangle)$ if and only if 
$q\leqslant \lfloor\frac{q}{a}\rfloor b$. 
Otherwise, the Geil-Matsumoto bound always gives an improvement with respect to Lewittes's bound.
We want to generalize these results to semigroups with any number of generators.

\begin{theorem}
\label{6:t:coincidences}
Let $\Lambda=\langle \lambda_1,\dots,\lambda_n\rangle$ with $\lambda_1<\lambda_i$ for all $i>1$.
The next statements are equivalent
\begin{enumerate}
\item
$GM_q(\Lambda)=L_q(\Lambda)$,
\item
$\Lambda\setminus\cup_{i=1}^n(q\lambda_i+\Lambda)=\Lambda\setminus(q\lambda_1+\Lambda)$,
\item
$q(\lambda_i-\lambda_1)\in\Lambda$ for all $i>1$.
\end{enumerate}
\end{theorem}

\begin{proof}
By Lemma~\ref{6:l:hlpgm} it is obvious that 2 implies 1.
The converse follows from the inclusion
$\Lambda\setminus\cup_{i=1}^n(q\lambda_i+\Lambda)\subseteq\Lambda\setminus(q\lambda_1+\Lambda)$ and the equality 
$GM_q(\Lambda)=L_q(\Lambda)$ which, by Lemma~\ref{6:l:hlpgm}, implies that
$\#(\Lambda\setminus\cup_{i=1}^n(q\lambda_i+\Lambda))=\#(\Lambda\setminus(q\lambda_1+\Lambda))$.

For the equivalence of the last two statements notice that
$q(\lambda_i-\lambda_1)\in\Lambda$ for all $i>1$
$\Longleftrightarrow$
$q\lambda_i\in q\lambda_1+\Lambda$ for all $i>1$
$\Longleftrightarrow$
$q\lambda_i+\Lambda\subseteq q\lambda_1+\Lambda$ for all $i>1$
$\Longleftrightarrow$
$\Lambda\setminus\cup_{i=1}^n(q\lambda_i+\Lambda)=\Lambda\setminus(q\lambda_1+\Lambda)$.
\end{proof}

Notice that under the hypothesis $q\in\Lambda$ then $q(\lambda-\lambda_1)\in\Lambda$ is satisfied by all $\lambda\in\Lambda$. So, Theorem~\ref{6:t:coincidences} generalizes Beelen-Ruano's result.

Theorem~\ref{6:t:coincidences}
suggests to analyze under what conditions 
$q(\lambda_i-\lambda_1)\in\Lambda$ for some $i>1$.
Let us first see in what cases 
$q(\lambda_i-\lambda_1)\in\{x\lambda_1+y\lambda_i:x,y\in{\mathbb N}_0\}$.
Notice that if $\gcd(\lambda_1,\lambda_i)=d$ then  
$\{x\lambda_1+y\lambda_i:x,y\in{\mathbb N}_0\}=d\langle\frac{\lambda_1}{d},\frac{\lambda_i}{d}\rangle$,
where by 
$d\langle\frac{\lambda_1}{d},\frac{\lambda_i}{d}\rangle$ we mean the set
$\{d\lambda:\lambda\in \langle\frac{\lambda_1}{d},\frac{\lambda_i}{d}\rangle\}$.
Obviously, 
$d\langle\frac{\lambda_1}{d},\frac{\lambda_i}{d}\rangle\subseteq\Lambda$.

\begin{lemma}
\label{6:l:qd}
Let $\gcd(\lambda_1,\lambda_i)=d$. Then  
$q(\lambda_i-\lambda_1)\in d\langle\frac{\lambda_1}{d},\frac{\lambda_i}{d}\rangle$ if and only if
$qd\leqslant\lfloor \frac{qd}{\lambda_1}\rfloor\lambda_i$.
In particular, if
$q\leqslant\lfloor \frac{q}{\lambda_1}\rfloor\lambda_i$
then $q(\lambda_i-\lambda_1)\in d\langle\frac{\lambda_1}{d},\frac{\lambda_i}{d}\rangle$.
\end{lemma}

\begin{proof}
We need to prove that
$q(\frac{\lambda_i}{d}-\frac{\lambda_1}{d})\in \langle\frac{\lambda_1}{d},\frac{\lambda_i}{d}\rangle$ if and only if
$qd\leqslant\lfloor \frac{qd}{\lambda_1}\rfloor\lambda_i$.
Suppose that $c$ is the inverse of $\frac{\lambda_i}{d}$ modulo $\frac{\lambda_1}{d}$.
By Lemma~\ref{6:lemma:membership},
$q(\frac{\lambda_i}{d}-\frac{\lambda_1}{d})
\in \langle\frac{\lambda_1}{d},\frac{\lambda_i}{d}\rangle$ if and only if
$\frac{\lambda_i}{d}(q(\frac{\lambda_i}{d}-\frac{\lambda_1}{d}) c\mod\frac{\lambda_1}{d})\leqslant q(\frac{\lambda_i}{d}-\frac{\lambda_1}{d})$, that is, 
$\frac{\lambda_i}{d}(q\mod\frac{\lambda_1}{d})\leqslant q(\frac{\lambda_i}{d}-\frac{\lambda_1}{d})$ which is equivalent to 
$qd\leqslant\lfloor \frac{qd}{\lambda_1}\rfloor\lambda_i$.

Now, if $q\leqslant\lfloor \frac{q}{\lambda_1}\rfloor\lambda_i$, then  
$qd\leqslant \lfloor\frac{q}{\lambda_1}\rfloor d\lambda_i
\leqslant \lfloor\frac{qd}{\lambda_1}\rfloor \lambda_i$ and the last statement follows.
\end{proof}

\begin{proposition}
\label{p:coincLGM}
Suppose $\lambda_1< \lambda_2< \dots< \lambda_n$ and let
$\Lambda=\langle \lambda_1, \lambda_2, \dots, \lambda_n\rangle$.
If $q\leqslant \lfloor\frac{q}{\lambda_1}\rfloor \lambda_2$ 
then $GM_q(\Lambda)=L_q(\Lambda)$.
\end{proposition}

\begin{proof}
By hypothesis, 
$q\leqslant \lfloor\frac{q}{\lambda_1}\rfloor \lambda_i$
for all $i>1$.
By Lemma~\ref{6:l:qd},
$q(\lambda_i-\lambda_1)\in \Lambda$ for all $i>1$ and by Theorem~\ref{6:t:coincidences},
$GM_q(\Lambda)=L_q(\Lambda)$.
\end{proof}

\begin{remark}
As mentioned, the converse is true when restricted to semigroups with two generators. Otherwise the converse is not true in general. 
For instance, consider $\Lambda=\langle 5,7,18\rangle $ with $q=9$.
We have $\Lambda=\{0, 5, 7, 10, 12, 14, 15, 17, 18, \dots\}$ and 
%
%
$\Lambda\setminus\cup_{\lambda_i\mbox{ generator of }\Lambda}(q\lambda_i+\Lambda)=$
$
\{0, 5, 7, 10, 12, 14, 15, 17, 18, 19, 20,$
$
21, 22, 23,
$ $24, 25, 26, 27, 28, 29, 30, 31, 32, 33, 34, 35, 36, 37, 38, 39, 40, 41, 42, 43,$
$
44, 46, 47, 48,$ $49, 51, 53, 54, 56, 58, 61\}=\Lambda\setminus(q\lambda_1+\Lambda).$
%
So $GM_q(\langle 5,7,18\rangle)=L_q(\langle 5,7,18\rangle)=46$.
However, $q (=9)>\lfloor\frac{q}{\lambda_1}\rfloor \lambda_2(=7)$.
The reason is that although $q(\lambda_2-\lambda_1)\not\in\langle \lambda_1,\lambda_2\rangle$, it holds that $q(\lambda_2-\lambda_1)\in\langle \lambda_1,\lambda_2,\lambda_3\rangle=\Lambda$.
\end{remark}

In Table~\ref{t:LGM}, together with the portion of semigroups for which 
the Lewittes and the Geil-Matsumoto bounds coincide,
we give the portion of semigroups satisfying the hypothesis in Proposition~\ref{p:coincLGM}.
From that table it is easy to check again that in general the converse of 
Proposition~\ref{p:coincLGM} is not true.

\begin{table}
\begin{center}
\begin{tabular}{|| c || c|c|c|c|c || c|c|c|c|c ||}
\hline
&\multicolumn{5}{c||}{Lewittes  = Geil-Matsumoto
}&\multicolumn{5}{c||}{$q\leqslant \lfloor\frac{q}{\lambda_1}\rfloor
\lambda_2$}\\\hline
Genus & q=2 & q=3 & q=9 & q=16 & q=256 & q=2 & q=3 & q=9 & q=16 & q=256\\\hline
2 & 50.00$\%$ & 100$\%$ & 100$\%$ & 100$\%$ & 100$\%$ & 50.00$\%$ & 100$\%$ & 100$\%$ & 100$\%$ & 100$\%$ \\ 
3 & 25.00$\%$ & 75.00$\%$ & 100$\%$ & 100$\%$ & 100$\%$ & 25.00$\%$ & 75.00$\%$ & 100$\%$ & 100$\%$ & 100$\%$ \\ 
4 & 42.86$\%$ & 57.14$\%$ & 100$\%$ & 100$\%$ & 100$\%$ & 14.29$\%$ & 42.86$\%$ & 85.71$\%$ & 100$\%$ & 100$\%$ \\ 
5 & 33.33$\%$ & 41.67$\%$ & 91.67$\%$ & 100$\%$ & 100$\%$ & 8.33$\%$ & 25.00$\%$ & 58.33$\%$ & 91.67$\%$ & 100$\%$ \\ 
6 & 21.74$\%$ & 43.48$\%$ & 86.96$\%$ & 100$\%$ & 100$\%$ & 4.35$\%$ & 17.39$\%$ & 43.48$\%$ & 82.61$\%$ & 100$\%$ \\ 
7 & 17.95$\%$ & 41.03$\%$ & 87.18$\%$ & 100$\%$ & 100$\%$ &  2.56$\%$ & 10.26$\%$ & 38.46$\%$ & 84.62$\%$ & 100$\%$ \\ 
8 & 14.93$\%$ & 37.31$\%$ & 85.07$\%$ & 100$\%$ & 100$\%$ & 1.49$\%$ & 5.97$\%$ & 53.73$\%$ & 91.04$\%$ & 100$\%$ \\ 
9 & 11.02$\%$ & 33.05$\%$ & 88.14$\%$ & 98.31$\%$ & 100$\%$ & 0.85$\%$ & 4.24$\%$ & 72.03$\%$ & 87.29$\%$ & 100$\%$ \\ 
10 & 8.82$\%$ & 29.90$\%$ & 88.24$\%$ & 95.59$\%$ & 100$\%$ & 0.49$\%$ & 2.45$\%$ & 79.90$\%$ & 78.92$\%$ & 100$\%$ \\ 
11 & 7.58$\%$ & 25.95$\%$ & 84.55$\%$ & 92.71$\%$ & 100$\%$ & 0.29$\%$ & 1.46$\%$ & 78.13$\%$ & 65.89$\%$ & 100$\%$ \\ 
12 & 6.59$\%$ & 23.48$\%$ & 78.89$\%$ & 90.88$\%$ & 100$\%$ & 0.17$\%$ & 1.01$\%$ & 69.93$\%$ & 54.05$\%$ & 100$\%$ \\ 
13 & 5.69$\%$ & 21.48$\%$ & 73.73$\%$ & 89.81$\%$ & 100$\%$ & 0.10$\%$ & 0.60$\%$ & 59.64$\%$ & 42.76$\%$ & 100$\%$ \\ 
14 & 5.02$\%$ & 18.90$\%$ & 69.76$\%$ & 88.66$\%$ & 100$\%$ & 0.06$\%$ & 0.35$\%$ & 49.26$\%$ & 33.73$\%$ & 100$\%$ \\ 
15 & 4.10$\%$ & 16.63$\%$ & 66.26$\%$ & 87.68$\%$ & 100$\%$ & 0.04$\%$ & 0.25$\%$ & 39.38$\%$ & 28.35$\%$ & 100$\%$ \\ 
16 & 3.45$\%$ & 14.77$\%$ & 63.23$\%$ & 87.22$\%$ & 100$\%$ & 0.02$\%$ & 0.15$\%$ & 30.86$\%$ & 28.67$\%$ & 100$\%$ \\ 
17 & 2.92$\%$ & 13.10$\%$ & 60.66$\%$ & 87.00$\%$ & 100$\%$ & 0.01$\%$ & 0.09$\%$ & 23.79$\%$ & 35.23$\%$ & 100$\%$ \\ 
18 & 2.38$\%$ & 11.66$\%$ & 58.74$\%$ & 87.03$\%$ & 100$\%$ & 0.01$\%$ & 0.06$\%$ & 18.33$\%$ & 45.70$\%$ & 100$\%$ \\ 
19 & 1.93$\%$ & 10.40$\%$ & 57.06$\%$ & 86.71$\%$ & 100$\%$ & 0.00$\%$ & 0.04$\%$ & 13.93$\%$ & 55.89$\%$ & 100$\%$ \\ 
20 & 1.60$\%$ & 9.28$\%$ & 55.71$\%$ & 85.43$\%$ & 100$\%$ & 0.00$\%$ & 0.02$\%$ & 10.55$\%$ & 62.47$\%$ & 99.95$\%$ \\ 
21 & 1.31$\%$ & 8.34$\%$ & 54.67$\%$ & 83.03$\%$ & 100$\%$ & 0.00$\%$ & 0.01$\%$ & 7.93$\%$ & 64.51$\%$ & 99.75$\%$ \\ 
22 & 1.09$\%$ & 7.48$\%$ & 53.95$\%$ & 80.14$\%$ & 100$\%$ & 0.00$\%$ & 0.01$\%$ & 5.93$\%$ & 62.93$\%$ & 99.19$\%$ \\ 
23 & 0.90$\%$ & 6.70$\%$ & 53.29$\%$ & 77.41$\%$ & 100$\%$ & 0.00$\%$ & 0.01$\%$ & 4.39$\%$ & 59.00$\%$ & 98.09$\%$ \\ 
24 & 0.75$\%$ & 6.02$\%$ & 52.46$\%$ & 75.16$\%$ & 100$\%$ & 0.00$\%$ & 0.00$\%$ & 3.25$\%$ & 53.67$\%$ & 96.50$\%$ \\ 
25 & 0.63$\%$ & 5.42$\%$ & 51.33$\%$ & 73.37$\%$ & 100$\%$ & 0.00$\%$ & 0.00$\%$ & 2.38$\%$ & 47.63$\%$ & 94.73$\%$ \\ 
26 & 0.53$\%$ & 4.90$\%$ & 49.94$\%$ & 71.94$\%$ & 100$\%$ & 0.00$\%$ & 0.00$\%$ & 1.74$\%$ & 41.35$\%$ & 93.12$\%$ \\ 
27 & 0.45$\%$ & 4.45$\%$ & 48.39$\%$ & 70.75$\%$ & 100$\%$ & 0.00$\%$ & 0.00$\%$ & 1.27$\%$ & 35.24$\%$ & 91.84$\%$ \\ 
28 & 0.38$\%$ & 4.07$\%$ & 46.81$\%$ & 69.73$\%$ & 100$\%$ & 0.00$\%$ & 0.00$\%$ & 0.92$\%$ & 29.58$\%$ & 90.87$\%$ \\ 
29 & 0.32$\%$ & 3.74$\%$ & 45.25$\%$ & 68.76$\%$ & 100$\%$ & 0.00$\%$ & 0.00$\%$ & 0.67$\%$ & 24.52$\%$ & 90.06$\%$ \\ 
30 & 0.27$\%$ & 3.44$\%$ & 43.76$\%$ & 67.80$\%$ & 100$\%$ & 0.00$\%$ & 0.00$\%$ & 0.48$\%$ & 20.12$\%$ & 89.25$\%$ \\
\hline 
\end{tabular}

\caption{Portion of semigroups for which the Lewittes and the Geil-Matsumoto bounds coincide and portion of semigroups satisfying the hypothesis in Proposition~\ref{p:coincLGM}, that is $q\leqslant \lfloor\frac{q}{\lambda_1}\rfloor \lambda_2$, where $\lambda_1,\lambda_2$ are the first and second smallest generators.}
\label{t:LGM}
\end{center}
\end{table}

\section{Simplifying the computation}
\label{6:sec:simplify} 

Next we investigate in which cases the computation of
$\Lambda\setminus\cup_{\lambda_i\mbox{ generator of }\Lambda}(q\lambda_i+\Lambda)$
can be simplified to
the computation of 
$\Lambda\setminus\cup_{i\in I}(q\lambda_i+\Lambda)$ 
for some index set $I$ 
smaller than the number of generators of $\Lambda$.
The next proposition can be proved very similarly as we proved Theorem~\ref{6:t:coincidences}.

\begin{proposition}
\label{6:p:coincidencesI}
Let $\Lambda=\langle \lambda_1,\dots,\lambda_n\rangle$ and let $I$ be an index set included in  $\{1, \dots, n\}$.
The next statements are equivalent.
\begin{enumerate}
\item
$\Lambda\setminus\cup_{i=1}^n(q\lambda_i+\Lambda)=\Lambda\setminus\cup_{i\in I}(q\lambda_i+\Lambda)$.
\item
For all $i\not\in I$ there exists 
$1\leqslant j\leqslant n$, $j\in I$ such that $q(\lambda_i-\lambda_j)\in\Lambda$.
\end{enumerate}
\end{proposition}

One consequence of
Proposition~\ref{6:p:coincidencesI} is the next proposition.

\begin{proposition}
\label{6:p:notots}
Let $\Lambda=\langle \lambda_1,\dots,\lambda_n\rangle$ with $\lambda_1<\lambda_2<\dots<\lambda_n$ and $\lambda_1<q$.
\begin{enumerate}
\item 
Let $\lambda_j$ be the maximum generator strictly smaller than $\frac{q}{\lfloor\frac{q}{\lambda_1}\rfloor}$ then\break
$\Lambda\setminus\cup_{i=1}^n(q\lambda_i+\Lambda)=\Lambda\setminus\cup_{i=1}^j(q\lambda_i+\Lambda)$.
\item
Let $\lambda_j$ be the maximum generator strictly smaller than $2\lambda_1-1$ then\break
$\Lambda\setminus\cup_{i=1}^n(q\lambda_i+\Lambda)=\Lambda\setminus\cup_{i=1}^j(q\lambda_i+\Lambda)$.
\end{enumerate}
\end{proposition}

\begin{proof}
The first statement is a consequence of Lemma~\ref{6:l:qd} together with Proposition~\ref{6:p:coincidencesI}.
For the second statement suppose that $q=x\lambda_1+y$ with $x,y$ integers and $x\geqslant 1$. Then 
$\frac{q}{\lfloor\frac{q}{\lambda_1}\rfloor}=\lambda_1+\frac{y}{x}$.
The result follows from the inequalities $x\geqslant 1$ and $y\leqslant \lambda_1-1$.
\end{proof}

We will call the generators that are strictly smaller than $2\lambda_1-1$ Geil-Matsumoto generators. What the last statement of the previous proposition says is that for computing the Geil-Matsumoto bound we only need to subtract from $\Lambda$
the sets $q\mu+\Lambda$ for $\mu$ a Geil-Matsumoto generator. Since in general we need to subtract these sets for {\em all} generators, this constitutes an improvement in terms of computation. In Table~\ref{t1},
we give the mean of the number of Geil-Matsumoto generators and non-Geil-Matsumoto generators per semigroup for different genera.
In Table~\ref{t2}, we give
the portion of Geil-Matsumoto generators (and non-Geil-Matsumoto generators) with respect to the total number of generators for different genera. We observe that, although the portion of non-Geil-Matsumoto generators decreases with the genus, it remains still significant, with a portion of more than 30$\%$ for genus 25.

\begin{table}
\begin{center}
\begin{tabular}{|c|c|c|} 
\hline 
&&\\ 
 Genus & 
\begin{tabular}{l}Mean of \\the number of \\GM generators \\per semigroup \end{tabular}& 
\begin{tabular}{l}Mean of \\the number of \\non-GM generators \\per semigroup \end{tabular}
\\
&&\\ 
\hline 
2&1.50&1.00\\ 
3&1.75&1.00\\ 
4&2.00&1.14\\ 
5&2.33&1.42\\ 
6&2.52&1.43\\ 
7&2.79&1.62\\ 
8&3.07&1.76\\ 
9&3.32&1.89\\ 
10&3.57&2.00\\ 
11&3.85&2.17\\ 
12&4.10&2.27\\ 
13&4.38&2.41\\ 
14&4.65&2.53\\ 
15&4.92&2.65\\ 
16&5.20&2.76\\ 
17&5.48&2.88\\ 
18&5.76&2.98\\ 
19&6.05&3.09\\ 
20&6.35&3.20\\ 
21&6.64&3.30\\ 
22&6.94&3.40\\ 
23&7.24&3.50\\ 
24&7.55&3.59\\ 
25&7.86&3.68\\ 
26&8.17&3.77\\ 
27&8.49&3.86\\ 
28&8.81&3.94\\ 
29&9.13&4.03\\ 
30&9.46&4.10\\
\hline
\end{tabular}
\caption{Mean of the number of Geil-Matsumoto generators and non-Geil-Matsumoto generators per semigroup}
\label{t1}
\end{center}
\end{table}

\begin{table}
\begin{center}
\begin{tabular}{|c|c|c|c|} 
\hline 
&&& \\ 
 Genus & 
\begin{tabular}{l}Total number of \\GM generators \\divided by the \\total number of \\generators\end{tabular} &
\begin{tabular}{l}Total number of \\non-GM generators \\divided by the \\total number of \\generators\end{tabular} &
\begin{tabular}{l}Mean of the \\portion of \\non-GM generators\\per semigroup\end{tabular}\\ 
&&&\\ 
\hline 
2&60.00$\%$ &40.00$\%$&41.67$\%$\\ 
3&63.64$\%$ &36.36$\%$&35.42$\%$\\ 
4&63.64$\%$ &36.36$\%$&38.57$\%$\\ 
5&62.22$\%$ &37.78$\%$&40.14$\%$\\ 
6&63.74$\%$ &36.26$\%$&37.43$\%$\\ 
7&63.37$\%$ &36.63$\%$&39.13$\%$\\ 
8&63.58$\%$ &36.42$\%$&39.03$\%$\\ 
9&63.74$\%$ &36.26$\%$&38.58$\%$\\ 
10&64.03$\%$ &35.97$\%$&38.39$\%$\\ 
11&63.96$\%$ &36.04$\%$&38.76$\%$\\ 
12&64.34$\%$ &35.66$\%$&38.26$\%$\\ 
13&64.54$\%$ &35.46$\%$&38.17$\%$\\ 
14&64.75$\%$ &35.25$\%$&37.99$\%$\\ 
15&65.01$\%$ &34.99$\%$&37.73$\%$\\ 
16&65.30$\%$ &34.70$\%$&37.45$\%$\\ 
17&65.56$\%$ &34.44$\%$&37.21$\%$\\ 
18&65.88$\%$ &34.12$\%$&36.87$\%$\\ 
19&66.19$\%$ &33.81$\%$&36.55$\%$\\ 
20&66.49$\%$ &33.51$\%$&36.25$\%$\\ 
21&66.79$\%$ &33.21$\%$&35.93$\%$\\ 
22&67.11$\%$ &32.89$\%$&35.59$\%$\\ 
23&67.43$\%$ &32.57$\%$&35.26$\%$\\ 
24&67.76$\%$ &32.24$\%$&34.91$\%$\\ 
25&68.08$\%$ &31.92$\%$&34.56$\%$\\ 
26&68.41$\%$ &31.59$\%$&34.21$\%$\\ 
27&68.74$\%$ &31.26$\%$&33.86$\%$\\ 
28&69.07$\%$ &30.93$\%$&33.50$\%$\\ 
29&69.40$\%$ &30.60$\%$&33.14$\%$\\ 
30&69.74$\%$ &30.26$\%$&32.77$\%$\\ 
\hline
\end{tabular}
\caption{Portion of Geil-Matsumoto generators}
\label{t2}
\end{center}
\end{table}

Proposition~\ref{6:p:notots} is a first consequence of 
Proposition~\ref{6:p:coincidencesI} and it
can be used to simplify the computation of the
Geil-Matsumoto bound.
We leave it as a problem for future research to find other consequences of 
Proposition~\ref{6:p:coincidencesI} to get further simplifications.

\section*{Acknowledgments}

The authors are grateful to Diego Ruano and also to the anonymous 
referees for many interesting comments.
This work was partly supported by
the Spanish Government through projects TIN2009-11689
``RIPUP'' and CONSOLIDER INGENIO 2010 CSD2007-00004 ``ARES'',
 and by the Government of Catalonia under grant 2009 SGR 1135.


\begin{thebibliography}{1}

\bibitem{6:BeelenRuano}
P. Beelen and D. Ruano.
\newblock Bounding the number of points on a curve using a generalization of Weierstrass semigroup.
\newblock {\em Designs, Codes and Cryptography}, Accepted, 2012.

\bibitem{6:GarciaRosales:intervals}
P.~A. Garc{\'{\i}}a-S{\'a}nchez and J.~C. Rosales.
\newblock Numerical semigroups generated by intervals.
\newblock {\em Pacific J. Math.}, 191(1):75--83, 1999.

\bibitem{6:Geil:normtrace}
Olav Geil.
\newblock On codes from norm-trace curves.
\newblock {\em Finite Fields Appl.}, 9(3):351--371, 2003.

\bibitem{6:GeilMatsumoto}
Olav Geil and Ryutaroh Matsumoto.
\newblock Bounding the number of {${\mathbb F}_q$}-rational places in algebraic
  function fields using {W}eierstrass semigroups.
\newblock {\em J. Pure Appl. Algebra}, 213(6):1152--1156, 2009.

\bibitem{6:HoLiPe}
Tom H{\o}holdt, Jacobus~H. van Lint, and Ruud Pellikaan.
\newblock Algebraic Geometry Codes.
\newblock In {\em Handbook of coding theory, {V}ol. {I}, {II}}, pages 871--961.
  North-Holland, Amsterdam, 1998.

\bibitem{6:KiPe}
Christoph Kirfel and Ruud Pellikaan.
\newblock The minimum distance of codes in an array coming from telescopic
  semigroups.
\newblock {\em IEEE Trans. Inform. Theory}, 41(6, part 1):1720--1732, 1995.
\newblock Special issue on algebraic geometry codes.

\bibitem{6:Lewittes}
Joseph Lewittes.
\newblock Places of degree one in function fields over finite fields.
\newblock {\em J. Pure Appl. Algebra}, 69(2):177--183, 1990.

\bibitem{6:RosalesGarciaSanchez}
J.~C. Rosales and P.~A. Garc{\'{\i}}a-S{\'a}nchez.
\newblock {\em Numerical semigroups}, volume~20 of {\em Developments in
  Mathematics}.
\newblock Springer, New York, 2009.

\bibitem{6:Stichtenoth}
Henning Stichtenoth.
\newblock {\em Algebraic function fields and codes}.
\newblock Universitext. Springer-Verlag, Berlin, 1993.

\end{thebibliography}

\end{document}